\documentclass[12pt,reqno]{amsart}
\usepackage{amsfonts}
\usepackage{amssymb}
\usepackage{amsmath}
\usepackage{latexsym}
\usepackage{amsthm}
\usepackage{epsfig}
\usepackage{color}
\usepackage{amsfonts,amssymb,mathrsfs,amscd}
\usepackage{comment}

\usepackage{amssymb,amsmath}
\usepackage{color}

\def\Bbb{\mathbb}

\def\Dt{\partial_t}
\def\eb{\varepsilon}

\def\R {\mathbb{R}}

\def\<{\left<}
\def\>{\right>}

\def\Nx{\nabla_x}
\def\Dx{\Delta_x}
\def\divv{\operatorname{div}}
\def\({\left(}
\def\){\right)}

\def\Bbb{\mathbb}

\newtheorem{proposition}{Proposition}[section]
\newtheorem{theorem}[proposition]{Theorem}
\newtheorem{corollary}[proposition]{Corollary}
\newtheorem{lemma}[proposition]{Lemma}

\theoremstyle{definition}
\newtheorem{definition}[proposition]{Definition}

\newtheorem{remark}[proposition]{Remark}

\numberwithin{equation}{section}

\def \no#1#2#3 {{\bf #1} (#3), #2.}
\def \eds#1#2#3 {#1, #2, #3.}

\title[Hyperbolic Navier-Stokes equations] {Hyperbolic relaxation of the 2D Navier-Stokes equations in a bounded domain }
\author[A. Ilyin, Yu. Rykov and S. Zelik]{ Alexei Ilyin${}^2$,  Yuri Rykov${}^2$ and Sergey Zelik${}^{1,2}$}
\address{${}^1$
University of Surrey, Department of Mathematics,
Guildford, GU2 7XH, United Kingdom,  s.zelik@surrey.ac.uk.}
\address{${}^2$ Keldysh Institute of Applied Mathematics,
Miusskaya sq. 4, 125047 Moscow, Russia}
\thanks{ 
A.I. and Yu. G. acknowledge financial support from the Russian Science Foundation (grant
no. 14-21-00025) and S.Z.'s research is supported by the Russian Science Foundation (grant
no. 14-41-00044) and the  RFBR grant 15-01-03587.
The authors would also like to thank Varga Kalantarov for many stimulating discussions}

\begin{document}
\begin{abstract} A hyperbolic relaxation of the classical Navier-Stokes problem in 2D bounded domain with Dirichlet boundary conditions is considered. It is proved that this relaxed problem possesses a global strong solution if the relaxation parameter is small and the appropriate norm of the initial data is not very large. Moreover, the dissipativity of such solutions is established and the singular limit as the relaxation parameter tends to zero is studied.
\end{abstract}
\maketitle
\tableofcontents
\def\R{\Bbb R}
\def\Dt{\partial_t}
\def\eb{\varepsilon}

\section{Introduction}\label{s1}
Various versions of hyperbolic  Navier-Stokes equations are of increasing current interest. For instance, such equations may appear from the  relaxation approximations of the Euler equations in the diffusive scaling limit:
\begin{equation}
\begin{cases}
\Dt u+\divv(U)+\Nx p=0, \ \ \ \divv u=g\\
\eb\Dt U+\nabla u+U=u\otimes u,
\end{cases}
\end{equation}
where $\eb>0$ and $U$ is a supplementary matrix valued variable. Excluding this variable, we end up with the following version of hyperbolic Navier-Stokes equations
\begin{equation}\label{0.HNS}
\eb\Dt^2 u+\Dt u+\divv(u\otimes u)+\Nx p=\Dx u+g,
\end{equation}
see \cite{bren,Hach,Raug,Sch} for the details.
\par
Another source of such equations is the theory of viscoelastic fluids. In particular, the equations
\begin{equation}\label{0,veNS}
\eb\Dt^2u+\eb\Dt\divv(u\otimes u)+\Dt u+\divv (u\otimes u)+\Nx p=\Dx u+g
\end{equation}
with the additional term $\eb\Dt\divv(u\otimes u)$ naturally arise in the theory of Jeffrey flows, see \cite{const,GGP,RAI,RAII}, see also references therein.
\par
We also report here on one more  motivation for  the hyperbolic relaxation of Navier-Stokes equations related with the computational aspects. From this point of view, the usefulness of hyperbolization of the Navier-Stokes equations can be described as follows. The usage of the explicit schemes, which can turn out to be the most adequate for flows with complex structure and very convenient for parallel computations, requires the time step  $\tau\sim h^2$ , where  $h$  is  the typical size of the spatial grid. In the hyperbolized version of Navier--Stokes system with small
parameter $\eb$    the time step is $\tau\sim h\sqrt{\eb}$   because of the hyperbolic nature of modified system. If we take $\eb$  of the order of $h$ , then there is a significant gain in the computation time. At the same time many physical systems has limited level of detailing, see, for example, \cite{Chet}. This fact leads to natural limitation for the space discretization in practical problems analogously to the situation with multi-phase models.  The estimates of possible influence of scales on the quality of numerical algorithms and the proximity estimates in the
 linear case can be found, for example, in \cite{Il-Ryk, Mysh-Tish, Rep-Chet}.
\par
The mathematical study of problem \eqref{0.HNS} in the case where $x\in\R^d$, $d=2$ or $3$ as well as for periodic boundary conditions are presented in \cite{bren,Raug}, see also \cite{RAI,RAII}.
\par
The main aim of the present paper is to study problem \eqref{0.HNS} in a bounded domain $\Omega\subset\R^2$ with Dirichlet boundary conditions. Note that, in contrast to the cases mentioned above, we cannot use the vorticity equation as it done in \cite{bren} or Strichartz estimates for wave operators which have been essentially used in \cite{Raug} (to the best of our knowledge nothing is known concerning the validity of Strichartz estimates for hyperbolic Stokes equations in bounded domains). By this reason, the have to work on the level of energy type estimates only and cannot use the approaches developed in \cite{bren} and \cite{Raug} at least in a direct way.
\par
One more principal difficulty related with these equations is that they do not possess any reasonable energy inequality for $\eb\ne0$, so to obtain the global existence of solutions we need to use the energy equality for $\eb=0$ and perturbation arguments. By this reason, our results are restricted to the case of small $\eb$ and cannot be extended to the case of arbitrarily large $\eb$. We expect that this restriction is not technical but is related with the nature of the considered problem. In order to support this point of view, we consider the simplified model of 1D hyperbolic Burgers equation
$$
\eb\Dt^2u+\Dt u+u\partial_x u=\partial^2_xu
$$
and prove that the solutions may blow up in finite time if the initial energy or/and $\eb>0$ is large enough, see Section \ref{s4.1} for the details.
\par
We will study the strong solutions $\xi_u(t):=\{u(t),\Dt u(t)\}$ of problem \eqref{0.HNS} which belong to the phase space
$$
\mathcal E^1_\eb:=\biggl\{\{u,v\}\in H^2(\Omega)\cap H^1_0(\Omega)\times H^1_0(\Omega),\ \divv u=\divv v=0\biggr\}
$$
endowed by the following norm:
$$
\|\xi_u\|^2_{\mathcal E^1_\eb}:=\eb\|\Dt u\|_{H^1}^2+\|\Dt u\|^2_{L^2}+\|u\|^2_{H^2}.
$$
The main result of the paper is a global existence and dissipativity of strong solutions of problem \eqref{0.HNS} if $\eb$ is small enough and the initial data satisfies
$$
\|\xi_u(0)\|_{\mathcal E^1_\eb}\le R(\eb),
$$
where the monotone decreasing function $R$ satisfies $\lim_{\eb\to0}R(\eb)=\infty$. As usual, the dissipativity means that
$$
\|\xi_u(t)\|_{\mathcal E^1_\eb}\le Q(\|\xi_u(t)\|_{\mathcal E^1_\eb})e^{-\alpha t}+Q(\|g\|_{L^2}),
$$
where the monotone function $Q$ and positive constant $\alpha$ are independent of $u$, $t$ and $\eb$, see Theorem \ref{Th1.main}.
\par
The paper is organized as follows.
\par
In Section \ref{s2}, we introduce the necessary spaces and notations which will be used throughout of the paper. Section \ref{s3} is devoted to the proof of the main result stated above.
\par
The singular limit $\eb\to0$ is studied in Section \ref{s.35}. In particular, we give there the results concerning the convergence of individual trajectories on finite time interval as well as the convergence of the corresponding global attractors.
\par
Finally, in Section \ref{s4}, we discuss possible extensions of proved results to the 3D case as well as to Jeffrey flows. In addition, the result concerning blow up in the hyperbolic Burgers equations is given there.

\section{Preliminaries}\label{s2}

In a bounded smooth domain $\Omega\subset\R^2$, we study the following problem:
\begin{equation}\label{1}
\begin{cases}
\eb\Dt^2 u+\Dt u+(u,\Nx)u+\Nx p=\Dx u+g,\\
\divv u=0,\ u\big|_{t=0}=u_0,\ \ \Dt u\big|_{t=0}=u_0',\\
u\big|_{\partial\Omega}=0.
\end{cases}
\end{equation}
Here $u=(u^1,u^2)$ and  $p$ are an unknown velocity vector field and pressure, respectively, and $g\in L^2(\Omega)$ is the given external force and $\eb>0$ is a given parameter which is assumed to be small enough.
\par
 As usual, we introduce the spaces $V$ and $\mathcal H$ as follows:
\begin{equation}\label{1.VH}
\aligned
&V:=\{u\in [H_0^1(\Omega)]^2,\ \ \divv u=0\},\\
&\mathcal H:=\{u\in [L^2(\Omega)]^2,\ \ \divv u=0,\ u\cdot n\big|_{\partial\Omega}=0\}
\endaligned
\end{equation}
and denote by $P:[L^2(\Omega)]^2\to\mathcal H$ the Leray--Helmholtz orthogonal projection onto solenoidal vector fields, see~\cite{TemNS}.
We also denote by $A:=-P\Delta$ the Stokes operator in $\mathcal H$ with Dirichlet boundary conditions. Then $A$ is a positive-definite self-adjoint operator in $\mathcal H$ with compact inverse with domain
$$
D(A)=[H^2(\Omega)\cap H_0^1(\Omega)]^2\cap \{\divv u=0\}.
$$
In addition, $D(A^{1/2})=V$ (with equality of norms), and
 $D(A^{-1/2})=D(A^{1/2})^*=V^*=H^{-1}(\Omega)$ (with equality of norms),
and we also set
$$
(u,v)_s:=(u,A^sv),\quad \text{for}\quad |s|\le1.
$$
\par
We now introduce the  natural energy spaces related with the hyperbolic relaxation \eqref{1} of the Navier-Stokes equations as follows:
\begin{equation}\label{1.ens}
 E^s:=D(A^{(s+1)/2})\times D(A^{s/2}), \ \ s\in\R,
\end{equation}
although we will use below only the cases where $s=-1,0,1$. The norms in these spaces are given by
\begin{equation}\label{2}
\|\xi_u\|_{\mathcal E^s_\eb}^2:=
\eb\|\Dt u\|_{D(A^{s/2})}^2+\|\Dt u\|^2_{D(A^{(s-1)/2})}+\|u\|_{D(A^{(s+1)/2})}^2,
\end{equation}
where
$$
\xi_u:=\{u,\Dt u\}.
$$
Note that these norms are equivalent for different positive values of the parameter $\eb$, but the dependence on $\eb$ is included to the definition of this norms in order to capture the right dependence of solutions on $\eb$ as $\eb\to0$.

The truncated norms
\begin{equation}\label{trunc}
\|\xi_u\|_{ E^s_\eb}^2:=
\eb\|\Dt u\|_{D(A^{s/2})}^2+\|u\|_{D(A^{(s+1)/2})}^2
\end{equation}
in these energy spaces will also  be useful in what follows. Furthermore,
all the technical estimates in Section~\ref{s3} below will be carried
out in terms of norm~\eqref{trunc} up to the last step when the full norm~\eqref{2}
is appended to the final estimate.

Finally, in order to exclude the pressure, we apply the Leray operator $P$ to both sides of equation \eqref{1} and get the equation for the velocity field $u$ only:
\begin{equation}\label{1.HNS}
\eb\Dt^2 u+\Dt u+P((u,\Nx)u)=-Au+g,\ \ \xi_u\big|_{t=0}=\{u_0,u_0'\},\ \ 
\end{equation}
where we assume for simplicity that $g=Pg$. Thus, by definition, a vector field $u=u(t,x)$ is a strong solution of the Navier-Stokes problem \eqref{1} on the interval $t\in[0,T]$ if
\begin{equation}\label{1.sol}
u\in C(0,T;D(A)),\ \ \Dt u\in C(0,T; V),\ \eb\Dt^2 u\in C(0,T;\mathcal H)
\end{equation}
and $u$ satisfies \eqref{1.HNS} as an equality in $\mathcal H$.
\par
For  $\eb=0$ the limiting equation is the
classical Navier-Stokes system
\begin{equation}\label{4}
\begin{cases}
\Dt v+(u,\Nx)v+\Nx p=\Dx v+g,\\
\divv v=0,\ v\big|_{t=0}=v_0,\ \\
\end{cases}
\end{equation}
which  has a unique strong solution $v$ and this solution possesses
the dissipative estimate in $H^2$:
\begin{equation}\label{1.dis0}
\|v(t)\|_{H^2}\le Q(\|v(0)\|_{H^2})e^{-\alpha t}+Q(\|g\|_{L^2}),
\end{equation}
where the monotone function $Q$ and positive constant $\alpha$ are independent of
$t$, $g$ and $v_0$, see \cite{BV}, \cite{Tem} and the references therein.
Moreover, differentiating equation \eqref{1} in time and arguing in a
standard way, we may obtain the corresponding estimates for the time derivatives of $v$, namely,
\begin{multline}\label{1.der}
\|\Dt v(t)\|_{L^2}+\|\Dt v\|_{L^2(t,t+1;V)}+\|\Dt^2 v\|_{L^2(t,t+1;V^*)}\le\\
Q(\|v(0)\|_{H^2})e^{-\alpha t}+Q(\|g\|_{L^2}).
\end{multline}
Since the derivation of this estimate from estimate \eqref{1.dis0} is
straightforward, we  leave it for the reader.

\begin{remark}
In what follows we shall be using the same notation
for different monotone increasing functions in dissipative
estimates like~\eqref{1.dis0}, \eqref{1.der}.
\end{remark}

\section{Key dissipative estimate}\label{s3}
The main aim of this section is to obtain the analogue of the dissipative estimate \eqref{1.dis0} for the case $\eb>0$. The main difference here is that, in contrast to the classical Navier-Stokes equations, we do not have basic energy identity if $\eb>0$. Moreover, we expect that the solutions of the perturbed problem may blow up in finite time if the initial energy is large enough, see Section \ref{s4} for more details. Thus, 
we may expect only that \eqref{1.dis0} remains true if $\eb>0$ is small or the initial energy is not very large. Namely, the following estimate
can be considered as a main result of the paper.

\begin{theorem}\label{Th1.main} Let the external forces $g\in \mathcal H$. Then, for every $R>0$, there exists $\eb_0=\eb_0(R)>0$ such that, for every $\eb\le\eb_0$ and every initial data $\xi_u(0)\in E^1$ satisfying $\|\xi_u(0)\|_{\mathcal E^s_\eb}\le R$, problem \eqref{1} possesses a unique global strong solution $\xi_u(t)\in E^1$ and the following dissipative estimate holds:
\begin{equation}\label{3}
\|\xi_u(t)\|_{\mathcal E^1_\eb}\le Q(\|\xi_u(0)\|_{\mathcal E_\eb^1})e^{-\alpha t}+Q(\|g\|_{L^2}),
\end{equation}
 where the positive constant $\alpha$ and monotone function $Q$ are
independent of $R$, $\eb\le\eb_0$, $t\ge0$ and the initial data $\xi_u(0)$
satisfying $\|\xi_u(0)\|_{\mathcal E_\eb^1}\le R$.
\end{theorem}
\begin{proof} We give below only the formal derivation of the dissipative estimate \eqref{3}. The existence a solution can then be  obtained in a standard way, for instance, by the Galerkin approximations, see \cite{BV}, \cite{TemNS}. The uniqueness of strong solutions is also straightforward and we do not discuss it here.
\par
The derivation of estimate \eqref{3} is based on the fact that a solution $u(t)$
of the perturbed equation (with small $\eb>0$), if it exists,
should be close (on a finite time interval $t\in[0,T]$)
to the corresponding limit Navier-Stokes system~\eqref{4}, which
corresponds to $\eb=0$). Thus, we take the strong solution $v=v(t)$ of
the Navier-Stokes problem~\eqref{4}
as an approximation of the corresponding solution $u=u(t)$ of the perturbed system.
Important for us is that
$$
v\big|_{t=0}=u\big|_{t=0}.
$$
Then, on the one hand, \eqref{4} is globally solvable and we have
\eqref{1.dis0} and \eqref{1.der}. 
On the other hand, we expect that the difference
\begin{equation}\label{uv}
w(t)=u(t)-v(t)
\end{equation}
is small in the appropriate norm which in turn would allow us to
verify the desired estimate using the perturbation arguments.
However, the realization of this scheme is not straightforward for
two reasons. The first one is that the passage from $\eb=0$ to $\eb>0$
is a {\it singular} perturbation and, in particular, a boundary
layer appears at $t=0$. The second, more technical one is that the
nonlinearity $(u,\Nx)u$ is {\it critical} from the point of view of
a hyperbolic equation even in the phase space $\mathcal E^1_\eb$.
So, the interplay between these too difficulties makes the
derivation of the desired estimate rather delicate and rather
technical. For the convenience of the reader, we will split it in
several steps.

\subsection*{Step 1. Estimates for $w$: weak norms} The function $w=u-v$  solves
the equation
\begin{equation}\label{7}
\begin{cases}
\eb\Dt^2 w\!+\!\Dt w\!+\!Aw\!=\\=-\!P\bigl[((w,\Nx)w)\!+\!(w,\Nx))v\!+\!((v,\Nx)w)\bigr]-\eb\Dt^2v,\\
\ w\big|_{t=0}=0,\ \ \Dt w\big|_{t=0}=u_0'-\Dt v(0),\\
w\big|_{\partial\Omega}=0.
\end{cases}
\end{equation}
We multiply this equation by $A^{-1}(\Dt w+\alpha w)$ where $\alpha>0$
and first consider the linear left-hand side. Integrating
by parts and taking into account~\eqref{trunc} we obtain
$$
\aligned
\bigl(\eb\Dt^2 w+\Dt w+&Aw,A^{-1}(\Dt w+\alpha w)\bigr)=\\=
&\frac12\frac d{dt}E^{-1}_\eb(\xi_w)+
(1-\varepsilon\alpha)\|\partial_t w\|_{H^{-1}}^2+\alpha\|w\|^2_H,
\endaligned
$$
where
$$
E^{-1}_\varepsilon(w):=
\|\xi_w\|^2_{E^{-1}_\varepsilon}
+\alpha\|w\|_{H^{-1}}^2+2\varepsilon\alpha(w,\partial_t w)_{-1}.
$$
Next, in view of the Poincar\'e inequality
\begin{equation}\label{Poincare}
\|w\|_{H^{-1}}^2\le\lambda_1^{-1}\|w\|^2_H,
\end{equation}
 taking $\alpha$ sufficiently small (and independent of  $\varepsilon$ as
$\varepsilon\to0$), we have
\begin{equation}\label{equivfunc}
c_1\|\xi_w\|^2_{{E}^{-1}_\varepsilon}\le {E}^{-1}_\varepsilon(\xi_w)\le c_2\|\xi_w\|_{{E}^{-1}_\varepsilon},
\end{equation}
where positive constants $c_1$ and $c_2$ are independent of $\eb$ and $\alpha$.
\par
Thus, the scalar product of~\eqref{7} and $A^{-1}(\Dt w+\alpha w)$
gives
\begin{equation}\label{8}
\aligned
&\frac d {dt} E_\eb^{-1}(\xi_w)+(1-\alpha\eb)\|\Dt w\|^2_{H^{-1}}+\alpha\|w\|^2_H\le\\
 &\ \ \le|((w,\Nx)w,A^{-1}(\Dt w+\alpha w))|+
 |((w,\Nx)v,A^{-1}(\Dt w+\alpha w))|+\\
 &\ \
 +|((v,\Nx)w,A^{-1}(\Dt w+\alpha w))|+\eb|(\Dt^2v,A^{-1}(\Dt w+\alpha w)|.
 \endaligned
\end{equation}
 Moreover, under this choice of $\alpha$, there exists $\beta>0$,
 which is also independent of $\eb\to0$, such that~\eqref{8} takes the following form
(here and below we write $E^{-1}_\eb(w)$ instead of $E_\eb^{-1}(\xi_w)$):
\begin{multline}\label{81}
\frac d {dt} E_\eb^{-1}(w)+\beta E^{-1}_\eb(w)+\beta \|\Dt w\|^2_{H^{-1}}\le
 |(\divv(w\otimes w),A^{-1}(\Dt w+\alpha w))|+\\+|(\divv(w\otimes v),A^{-1}(\Dt w+\alpha w))|+\\
 +|(\divv(v\otimes w),A^{-1}(\Dt w+\alpha w))|+\eb|(\Dt^2v,A^{-1}(\Dt w+\alpha w)|,
\end{multline}
where
$$
\divv (u\otimes v):=\sum_{i,j=1}^2\partial_i(u^iv^j)=(u,\Nx)v-v\divv u=(u,\Nx)v
$$
on divergent free vector fields.
So, we only need to estimate the terms on the right-hand side of this inequality.
The most difficult term here is the first one, and we are actually
unable to estimate in a closed form. Instead, using the regularity of
the Stokes operator and the interpolation inequality
$\|w\|_{L^4}^4\le C\|w\|_{H^2}\|w\|^3_{L^2}$, we can estimate it as follows:
\begin{multline*}\label{9}
|(\divv(w\otimes w),A^{-1}(\Dt w+\alpha w))|\le\\\le  C\|w\otimes w\|_{L^2}\|A^{-1}(\Dt w+\alpha w)\|_{H^1}\le
C\|w\|_{L^4}^2(\|\Dt w\|_{H^{-1}}+\|w\|_{H^{-1}})
\le\\\le C\|w\|_{H^{2}}\|w\|_{L^2}^3+\frac\beta 3(\|\Dt w\|^2_{H^{-1}}+\|w\|_{L^2}^2).
\end{multline*}
The second and third terms are estimated in the same way,
but using the fact that $\|v\|_{L^\infty}$ is under  control~\eqref{1.dis0}.
Using also~\eqref{Poincare} we obtain
\begin{multline*}
|(\divv(w\otimes v),A^{-1}(\Dt w+\alpha w))|+
|(\divv(v\otimes w),A^{-1}(\Dt w+\alpha w))|\le\\
\le C\||w|\cdot|v|\|_{L^2}(\|\Dt w\|_{H^{-1}}+\|w\|_{H^{-1}})\le
 C_R\|w\|^2_H+\frac\beta 3 \|w\|_{H^{-1}}^2.
\end{multline*}
For the last term in~\eqref{81} we have
$$
\eb|(\Dt^2v,A^{-1}(\Dt w+\alpha w)|\le
C\eb^2\|\Dt^2 v\|^2_{H^{-1}}+\frac\beta 3(\|\partial_t w\|_{H^{-1}}^2+\|w\|^2_H),
$$
and collecting the above estimates we finally obtain
\begin{equation}\label{10}
\frac d {dt} E_\eb^{-1}(w)\le \bigl(C_R+C\|w\|_{H^2}E_{\eb}^{-1}(w)^{1/2}\bigr) E_{\eb}^{-1}(w)+
C\eb^2\|\Dt^2 v\|^2_{H^{-1}},
\end{equation}
where the constant $C_R$ depends only on $R$
(recall that we assume that $\|\xi_u(0)\|_{E_\eb^1}\le R$). It is also important that
by the equivalence~\eqref{equivfunc}
\begin{equation}\label{11}
E_\eb^{-1}(w(0))\le C\eb\|\Dt u(0)\|_{H^{-1}}^2
\le C_R\eb.
\end{equation}
Unfortunately, estimate \eqref{10} is not enough to obtain an
estimate for $w$ due to the presence of the $H^2$-norm of $w$,
so we need more estimates to control it.
\par
\subsection*{Step 2. $H^2$-estimates for $u$.}
Recall that $u=v+w$ and the $H^2$-norm of $v$ is under  control,
so estimating the $H^2$-norm of $u$ is equivalent to estimating
the $H^2$-norm of $w$. However, it is more convenient  to
work with $u$ on this stage (due to the presence of the term
$\eb\Dt^2 v$ in the right-hand side of the equation for $w$
which requires too much regularity of the initial data to
be properly estimated). Multiplying the initial
equation \eqref{1} by $A(\Dt u+\alpha u)$ and arguing as before, we end up with
\begin{multline}\label{12}
\frac d{dt} E^1_\eb(u)+\beta E^1_\eb(u)+\beta\|\Dt u\|_{H^1}^2\le \\ \le
|((u,\Nx)u,A(\Dt u+\alpha u))|+|(g,A(\Dt u+\alpha u))|.
\end{multline}
The key problem is again to estimate the first term on the right-hand side.
 We write it as follows
$$
(u,\Nx)u=(w,\Nx) w+(v,\Nx v)+(v,\Nx) w+(w,\Nx)v
$$
and estimate each term separately. For simplicity we will estimate only the terms
with multiplication by $A\Dt u$ (the multiplication by $Au$ is analogous, but easier).
Integrating by parts  we have
\begin{multline}\label{13}
|((w,\Nx)w,A\Dt u)|\le C(\|w\|_{L^\infty}^2\|w\|_{H^2}^2+\|w\|_{W^{1,4}}^4)+
\frac\beta4\|\Dt u\|^2_{H^1}\le\\
\le C E_\eb^{-1}(w)^{1/2}[E^1_\eb(u)^{3/2}+\|v\|^3_{H^2}]+\frac\beta4\|\Dt u\|^2_{H^1},
\end{multline}
where we have  used the interpolation inequalities
$$
\|w\|_{L^\infty}^2\le C\|w\|_{L^2}\|w\|_{H^2},\quad
 \|w\|_{W^{1,4}}^4\le C\|w\|_{L^2}\|w\|_{H^2}^3
$$
and~\eqref{uv}. Since we have~\eqref{1.dis0}, the second term is straightforward
\begin{multline*}
|((v,\Nx)v,A\Dt u)|\le C\|v\|_{H^2}^2\|\Dt u\|_{H^1}\le\\
\le Q(\|g\|_{L^2})+Q(\|u(0)\|_{H^2})e^{-\alpha t}+\frac\beta4\|\Dt u\|^2_{H^1}.
\end{multline*}
The rest two terms can be estimated analogously to \eqref{13}:
\begin{multline*}
|((v,\Nx)w,A\Dt u)|+|(w,\Nx)v,A\Dt u)|\le
 C\|v\|_{H^2}^2\|w\|_{H^2}^2+\\+ \frac\beta4\|\Dt u\|^2_{H^1}\le
 \frac\beta4\|\Dt u\|^2_{H^1}+\\+
 \left(Q(\|g\|_{L^2})+Q(\|u(0)\|_{H^2}) e^{-\alpha t}\right)\| w\|^2_{H^2}\le\\
 \le \frac\beta4\|\Dt u\|^2_{H^1}+\left(Q'(\|g\|_{L^2})+Q'(\|u(0)\|_{H^2}) e^{-\alpha t}\right)
 (1+\|u\|_{H^2}^2)
\end{multline*}
with a different function $Q'$ with the same properties.
Inserting these estimates into \eqref{12} and omitting the primes, we get
\begin{multline}\label{15}
\frac d{dt} E^1_\eb(u)+\beta E^1_\eb(u)+\frac\beta 4\|\Dt u\|_{H^1}^2\le\\\le
 C E_\eb^{-1}(w)^{1/2}[E^1_\eb(u)^{3/2}+\|v\|^3_{H^2}]+\\+
\left(Q(\|g\|_{L^2})+Q(\|u(0)\|_{H^2}) e^{-\alpha t}\right)(1+\|u\|^2_{H^2}).
\end{multline}
Inequalities \eqref{10}, \eqref{11} and \eqref{15} are enough to
verify the global existence on any {\it finite} interval $t\in[0,T]$ if
$\eb\le \eb(T)$ is small enough, but still not sufficient to iterate the
estimates and get the global existence and dissipativity
(due to the presence of the big term $Q(\|g\|_{L^2})\|u\|_{H^2}^2$
which can not be absorbed by the term $\beta E^1_\eb(u)$ in the left-hand side).
To overcome this we need one more step involving the ``parabolic" type estimates for $A u$.

\subsection*{Step 3. Parabolic estimates for $Au$}
 We multiply equation \eqref{1} by $Au$ to get
\begin{multline}\label{16}
\frac d{dt}\(\eb(\Dt u,Au)+\frac12\|u\|_{H^1}^2\)+\|u\|^2_{H^2}-\eb\|\Dt u\|^2_{H^1}\le\\
\le |(u,\Nx) u, Au)|+|(g,Au)|.
\end{multline}
Estimating the non-linear term analogously to Step 2, we get
\begin{multline}\label{17}
\frac12\|u\|^2_{H^2}+\frac d{dt}\(\eb(\Dt u,Au)+\frac12\|u\|_{H^1}^2\)-\eb\|\Dt u\|^2_{H^1}\le\\\le
 C E_\eb^{-1}(w)^{1/2}[E^1_\eb(u)^{3/2}+
 \|v\|^3_{H^2}]+\\+Q(\|g\|_{L^2})+Q(\|u(0)\|^2_{H^2}e^{-\alpha t}.
\end{multline}
We now multiply inequality \eqref{17} by $2Q(\|g\|_{L^2})$ and add the
result to inequality \eqref{15}. This gives the following estimate
\begin{multline}\label{171}
\frac d{dt} \bar E^1_\eb(u)+\beta  E^1_\eb(u)+
\left(\frac\beta 4-2\eb Q(\|g\|_{L^2})\right)\|\Dt u\|_{H^1}^2
\le\\
\le C \bigl(1+2 Q(\|g\|_{L^2})\bigr)E_\eb^{-1}(w)^{1/2}[E^1_\eb(u)^{3/2}+\|v\|^3_{H^2}]+\\+
2Q(\|g\|_{L^2})^2+2Q(\|g\|_{L^2})Q(\|u(0)\|_{H^2})e^{-\alpha t}+
Q(\|u(0)\|_{H^2})\|u\|_{H^2}^2e^{-\alpha t},
\end{multline}
where
$$
\bar E^1_\eb(u):=E^1_\eb(u)+Q(\|g\|_{L^2})\left(\,2\eb(\Dt u,Au)+\|u\|^2_{H^1}\right).
$$
Thus, for sufficiently small $\eb$ we have
$$
C^{-1}E^1_\eb(u)\le \bar E^1_\eb(u)\le C E^1_\eb(u),
$$
where the constant $C=C_R$ is independent of $\eb$. Therefore, we may replace $E^1_\eb(u)$ by
$\bar E_\eb^1(u)$. Using the dissipative estimate~\eqref{1.dis0} for
the term $\|v\|_{H^2}^3$ in~\eqref{171} and dropping the bar-sign we obtain the final estimate:
\begin{multline}\label{18}
\frac d{dt}  E^1_\eb(u)+\left(\beta-Q(\|u(0)\|_{H^2})e^{-\alpha t}-
C[E^{-1}_\eb(w)]^{1/2}[ E^1_\eb(u)]^{1/2}\right)  E^1_\eb(u)\le\\\le C (1+E_\eb^{-1}(w)^{1/2})(Q(\|g\|_{L^2})+Q(\|u(0)\|_{H^2}) e^{-\alpha t}),
\end{multline}
where all of the constants are positive and are independent of $\eb\to0$.
As we will see below, this estimate together with inequalities \eqref{10} and \eqref{11} is
 sufficient to obtain the desired dissipative estimate for $u$.

\subsection*{Step 4. Completion of the proof}
Following \cite{Zel-DCDS} (see also \cite{Zel-PLMS} and \cite{Gra}), we first establish the dissipative estimate
on a {\it finite} time interval $t\in[0,T]$ where $T$ is large
enough. Namely, we claim that inequalities \eqref{10},\eqref{11}
and \eqref{18} imply the following intermediate result.

\begin{lemma}\label{Lem1.1} For every $R>0$ and every $T\in\R_+$ there exists
$\eb_0=\eb_0(T,R)$ such that for every $\eb<\eb_0$ and every initial data
$\xi_u(0)$ satisfying $\|\xi_u(0)\|_{ E^1_\eb}\le R$, there exists a unique solution $\xi_u(t)$, $t\in [0,T]$ of problem \eqref{1} satisfying the dissipative estimate
\begin{equation}\label{19}
\eb\|\Dt u(t)\|_{H^1}^2+\|u(t)\|^2_{H^2}\le Q_R(\|\xi_u(0)\|_{E^1_\eb})e^{-\alpha_R t}+Q(\|g\|_{L^2}),
\end{equation}
where the function $Q$ is independent of $T$, $R$ and $\eb$
and the functions $Q_R$ and $\alpha_R$ are independent of $T$ and $\eb$.
\end{lemma}
\begin{proof} Indeed, assume for the first that
\begin{equation}\label{1.bound}
C[E^{-1}_\eb(w)]^{1/2}[ E^1_\eb(u)]^{1/2}\le \frac12\beta,\ \ \ E^{-1}_\eb(w(t))\le 1,\ \  \ t\in[0,T].
\end{equation}
Then the Gronwall inequality applied to \eqref{18} gives us the desired estimate
\begin{equation}\label{1.good}
 E^1_\eb(u(t))\le \bar Q(\|g\|_{L^2})+\bar Q(\|\xi_u(0)\|_{E^1_\eb})e^{-\bar\beta t},
\end{equation}
for some monotone function $\bar Q$ and positive constant $\bar\beta$ which are independent of $t$ (but may depend on $R$). Recall that this estimate is obtained under the assumption \eqref{1.bound} and we still need to justify it. To this end, we note that estimate \eqref{1.good} gives us the following control
\begin{equation}
\|w(t)\|_{H^2}\le C( E^1_\eb(u(t))+C\|v\|_{H^2}^2)^{1/2}\le Q_R
\end{equation}
and, therefore \eqref{10} reads
\begin{equation}\label{1.weakest}
\frac d {dt} E_\eb^{-1}(w)\le \(C_R+Q_RE^{-1}_\eb(w)^{1/2}\) E_{\eb}^{-1}(w)+C\eb^2\|\Dt^2 v\|^2_{H^{-1}},
\end{equation}
This estimate, together with the fact that $E^{-1}_\eb(w(0))\le C_R\eb$ implies that, for sufficiently small $\eb\le \eb_0(T,R)$, the quantity $E^{-1}_\eb(w(t)))$ remains small and can be estimated as follows:
\begin{equation}
E_\eb^{-1}(w(t))\le Q_R\eb e^{(C_R+Q_R)T},\ \ t\in[0,T].
\end{equation}
Thus, we end up with the control
$$
C[E^{-1}_\eb(w(t))]^{1/2}E^1_\eb(u(t))]^{1/2}\le \bar Q_R\eb e^{\bar Q_RT},
$$
where the constant $Q_R$ is independent of $\eb\to0$.
Finally, since the bounds in~\eqref{1.bound} are satisfied for $t=0$ for sufficiently
small $\eb$, if we assume in addition that
 $\bar Q_R\eb e^{\bar Q_R T}\le \beta/2$, we  will get \eqref{1.bound}
 by continuity arguments. This finishes the proof of the lemma.
\end{proof}

Note that \eqref{19} still not enough to verify the global existence
 and dissipativity since the left-hand side does not contain
does the $L^2$-norm of $\Dt u$ without $\eb$, see definition~\eqref{2}.
 To overcome this problem and append $\|\Dt u\|_{L^2}$ to our estimate obtained in
the norm $E^1_\varepsilon$,  we recall  the boundary layer estimate for the second order ODE
\begin{equation}\label{20}
\eb\Dt^2 u+\Dt u=h(t),\ \ \|h(t)\|_{L^2}\le C(\|u(t)\|_{H^2}^2+\|g\|_{L^2}+1),
\end{equation}
Namely,
\begin{multline*}
\|\Dt u(t)\|_{L^2}\le \|\Dt u(0)\|_{L^2}e^{-t/\eb}+\frac1\eb\int_0^te^{-(t-s)/\eb}\|h(s)\|_{L^2}\,ds\le\\\le
\|\Dt u(0)\|_{L^2}e^{-t/\eb}+C\(1+\|g\|_{L^2}+\max_{s\in[0,t]}\|u(s)\|_{H^2}^2\).
\end{multline*}
Using this estimate on the interval $t\in[T-1,T]$ and assuming that $T\ge1$, we derive from \eqref{19} that
\begin{equation}\label{21}
\|\xi_u(t)\|_{\mathcal E^1_\eb}\le Q_R(\|\xi_u(0)\|_{\mathcal E^1_\eb})e^{-\alpha_R t}+Q(\|g\|_{L^2}),
\end{equation}
for $\|\xi_u(0)\|_{ \mathcal E^1_\eb}\le R$ and $1\le t\le T$. This estimate is already
enough to get the desired dissipative estimate and finish the proof of the theorem. Indeed,
 if we take $T=T(R)$ large enough, the $\mathcal E_\eb^1$-norm of the solution $u(t)$
 at point $t=T$ will be less than this norm at point $t=0$ according to \eqref{19} if $R$
 is large enough (say, if $R\ge 2Q(\|g\|_{L^2})$. This allows to iterate the procedure and
 obtain the global existence and estimate \eqref{3}. Thus, Theorem \ref{Th1.main} is proved.
\end{proof}

\section{Attractors and their singular limit as $\eb\to0$}\label{s.35}
In this section  we verify the existence of global attractors associated with the Navier-Stokes problem \eqref{1} and their convergence as $\eb\to0$ to the limit attractor associated with the classical Navier-Stokes problem. We first recall that the global well-posedness and dissipativity of problem \eqref{1} is established not for all initial data $\xi_u(0)\in E^1$ and not for all $\eb>0$, but only for relatively small $\eb>0$ and initial data satisfying the assumption
\begin{equation}\label{33.1}
\xi_u(0)\in B(0,R, \mathcal E^1_\eb):=\{\xi_0\in E^1,\ \ \|\xi_0\|_{\mathcal E^1_\eb}\le R\},
\end{equation}
where $R=R(\eb)$ is monotone increasing and satisfying
\begin{equation}
\lim_{\eb\to0}R(\eb)=\infty,
\end{equation}
see Theorem \ref{Th1.main}. By this reason, it looks natural to consider equation \eqref{1} in the phase space
\begin{equation}\label{33.phase}
\Phi_\eb:=\cup_{t\ge0}\bigl\{\xi_u(t),\ \ \xi_u(0)\in B(0,R,\mathcal E^1_\eb)\bigr\}.
\end{equation}
Then, according to Theorem \ref{Th1.main}, the solution semigroups
\begin{equation}
S_\eb(t)\xi_u(0):=\xi_u(t),\ \ S_\eb(t)\Phi_\eb\subset \Phi_\eb
\end{equation}
is well-defined and dissipative on $\Phi_\eb$ for $\eb>0$ being small enough. Moreover, these semigroups are continuous on $\Phi_\eb$ with respect to the initial data (in the topology of the space $E^1$) for every fixed $t$ and $\eb$ and, in particular, the set $\Phi_\eb$ is closed in $E^1$. Thus, we may speak about global attractors of semigroups $S_\eb(t)$ on $\Phi_\eb$. For the convenience of the reader, we remind the definition of a global attractor, see \cite{BV,Tem} for more details.

\begin{definition}\label{Def33.attr} A set $\mathcal A_\eb$ is a global attractor of the semigroup $S_\eb(t):\Phi_\eb\to\Phi_\eb$ if
\par
1) The set $\mathcal A_\eb$ is compact in $\Phi_\eb$;
\par
2) The set $\mathcal A_\eb$ is strictly invariant: $S_\eb(t)\mathcal A_\eb=\mathcal A_\eb$;
\par
3) The set $\mathcal A_\eb$ attracts the images of all bounded subsets of $\Phi_\eb$ as $t\to\infty$, i.e., for every bounded subset $B\subset \Phi_\eb$ and every neighbourhood $\mathcal O(\mathcal A_\eb)$ of the attractor $\mathcal A_\eb$, there exists $T=T(B,\mathcal O)$ such that
$$
S_\eb(t)B\subset \mathcal O(\mathcal A_\eb)
$$
if $t\ge T$. In the case $\eb>0$ the whole phase space $\Phi_\eb$ is bounded, so we may state and check the attraction property for $B=\Phi_\eb$ only.
\par
In the case where the phase space $\Phi_\eb$ is endowed by the {\it weak} or {\it strong} topology of the space $E^1$, we will refer to $\mathcal A_\eb$ as {\it weak} or {\it strong} attractor respectively.
\end{definition}
The most straightforward is the existence of a weak attractor, so we will start with establishing this fact.
\begin{proposition}\label{Prop33.weak} Let $\eb>0$ be small enough. Then the solution semigroup $S_\eb(t)$ acting on the phase space $\Phi_\eb$ defined above possesses a weak attractor $\mathcal A_\eb$ which satisfies the estimate
\begin{equation}\label{33.abound}
\|\mathcal A_\eb\|_{\mathcal E_\eb^1}\le\bar R,
\end{equation}
where $\bar R$ is independent of $\eb\to0$. As usual this attractor is generated by all bounded solutions of problem \eqref{1} defined for all $t$
\begin{equation}\label{33.aa}
\mathcal A_\eb=\mathcal K_\eb\big|_{t=0},
\end{equation}
where
\begin{multline}
\mathcal K_\eb:=\{\xi_u, \|\xi_u(t)\|_{\mathcal E^1_\eb}\le \bar R, \\ S_\eb(h)\xi_u(t)=\xi_u(t+h),\ \ h\ge0,\ t\in\R\}\subset C(\R,\Phi_\eb).
\end{multline}
\end{proposition}
\begin{proof} In order to prove the proposition, it is sufficient to verify two facts. Namely, that there exists a {\it compact} absorbing set for the semigroup $S_\eb(t)$ and that the semigroup is weakly continuous on it, see \cite{BV} for details. The first fact follows from the dissipative estimate proved in Theorem \ref{Th1.main}. Indeed, estimate \eqref{3} guarantees that the ball $B(0,\bar R, \mathcal E^1_\eb)$ will be an absorbing ball for this semigroup if, say, $\bar R=2Q(\|g\|_{L^2})$. This ball is weakly compact by Alaoglu theorem. The weak continuity is straightforward and standard, so we left its rigorous verification to the reader. Thus, the existence of a {\it weak} attractor $\mathcal A_\eb$ is verified and the representation formula \eqref{33.aa} also followed from the abstract attractor's existence theorem and the proposition is proved.
\end{proof}
We are now ready to verify that the constructed weak attractor is actually a strong one.

\begin{proposition}\label{Prop33.statr} Let $\eb>0$ be small enough. Then the solution semigroup $S_\eb(t):\Phi_\eb\to\Phi_\eb$ associated with the hyperbolic Navier-Stokes system \eqref{1} possesses a strong global attractor which coincides with the weak attractor constructed above.
\end{proposition}
\begin{proof} We will check the asymptotic compactness of the semigroup $S_\eb(t)$ using the so-called energy method. To this end, we take an arbitrary sequence of the initial data $\xi_n\in\Phi_\eb$ and arbitrary sequence $t_n\to\infty$ and need to verify that the sequence $\{S_\eb(t_n)\xi_n\}_{n=1}^\infty$ is precompact. Let $\xi_{u_n}(t):=S_\eb(t+t_n)\xi_n$, $t\ge-t_n$ be a sequence of solutions of problem \eqref{1} associated with this sequence. Extending it by zero for $t\le -t_n$, from the key dissipative estimate \eqref{3}, we infer that $\xi_{u_n}$ is uniformly bounded in the space $L^\infty(\R, E^1)$. Thus, without loss of generality, we may assume that
\begin{equation}\label{33.conv}
\xi_{u_n}\to\xi_u\ \ \text{weakly-star in }\ L^\infty_{loc}(\R,E^1).
\end{equation}
Moreover, utilizing the fact that $u_n$ solves \eqref{1}, we also get that $\xi_{u_n}(t)\rightharpoondown \xi_{u}(t)$ for every fixed $t\in\R$ as well as the fact that the limit function $\xi_u\in \mathcal K_\eb$. In particular,
\begin{equation}\label{weakxi}
\xi_{u_n}(0)\rightharpoondown\xi_u(0)
\end{equation}
and to verify the desired asymptotic compactness we only need to check that this convergence is in a fact strong. We will utilize the so-called energy method, see \cite{B,MRW} for more details. Namely, multiplying equation \eqref{1} by $\Dt Au+\alpha Au$, where $\alpha>0$ is a {\it big} number which will be determined later, we get
\begin{multline}
\frac d{dt}\frac12\(\|\xi_u\|^2_{E^1_\eb}+\alpha\|u\|^2_{H^1}+2\alpha\eb(\Dt u,Au)\)+\\+\beta \(\|\xi_u\|^2_{E^1_\eb}+\alpha\|u\|^2_{H^1}+2\alpha\eb(\Dt u,Au)\)+(1-(\alpha+\beta)\eb)\|\Dt u\|^2_{H^1}+\\+(\alpha-\beta)\|u\|^2_{H^2}-\alpha\beta\|u\|^2_{H^1}-2\eb\alpha\beta(\Dt u,Au)+\\+((u,\Nx)u,\Dt Au+\alpha Au)=(g,\Dt Au+\alpha Au),
\end{multline}
where $\beta>0$ is also a positive number. The validity of the energy identity for strong solutions can be verified in a standard way, say, by approximating the solution $u$ by $P_Nu$, see e.g. \cite{BV}. Denoting
\begin{multline}
E_\eb^1(\xi_u):=\|\xi_u\|^2_{E^1_\eb}+\alpha\|u\|^2_{H^1}+2\alpha\eb(\Dt u,Au),\\ N(\xi_u):=\alpha((u,\Nx)u, Au)-((\Dt u,\Nx)u,Au)-\\-((u,\Nx)\Dt u,Au)-\alpha(g,Au)
\end{multline}
and
$$
L(\xi_u)\!:=\!(1-(\alpha+\beta)\eb)\|\Dt u\|^2_{H^1}+(\alpha-\beta)\|u\|^2_{H^2}-\alpha\beta\|u\|^2_{H^1}-2\eb\alpha\beta(\Dt u,Au) ,
$$
we transform the identity as follows:
\begin{equation}
\frac12 \frac d{dt}E^1_\eb(\xi_u)+\beta E^1_\eb(\xi_u)+L(\xi_u)+N(\xi_u)=\frac d{dt}I(u),
\end{equation}
where $I(u):=(g, Au)-((u,\Nx)u,Au)$.
\par
Using the last identity for solutions $u_n$ instead of $u$, after integration in time we get
\begin{multline}
E^1_\eb(\xi_{u_n}(0))+2\int_{-t_n}^0e^{2\beta s}(L(\xi_{u_n}(s))+N(\xi_{u_n}(s))+\beta I(u_n(s)))\,ds=\\=2I(u_n(0))+2e^{-\beta t_n}(E^1_\eb(\xi_{u_n}(-t_n))-I(u_n(-t_n))).
\end{multline}
Our task now is to pass to the limit $n\to\infty$ in this identity. To this end, we first note that due to the boundedness of the sequence $\xi_n$ and the fact that $t_n\to\infty$, the terms containing $u_n(-t_n)$ vanish in the limit. Next, due to already proved weak convergence in $E^1$, the passage to the limit in terms containing the functional $I$ is also immediate. Let us now pass to the limit in terms containing $L(\xi)$ and $N(\xi)$. The only non-trivial term in $N(\xi_{u_n})$ is $((u_n,\Nx)\Dt u_n,Au_n)$ (for other terms of $N$ the proved weak convergence is enough to pass to the limit). We write this terms as follows:
$$
((u_n,\Nx)\Dt u_n,Au_n)=((u,\Nx)\Dt u_n,Au_n)+((u_n-u,\Nx)\Dt u_n,Au_n).
$$
Since the sequence $\xi_{u_n}$ is bounded in $E^1$ and $u_n\to u$ {\it strongly} in the space $C_{loc}(R\times\bar\Omega)$, the second term tends to zero and we only need to study the first one. Let us introduce the quadratic form
$$
Q(\xi_{u_n}):=L(\xi_{u_n})-((u,\Nx)\Dt u_n,Au_n).
$$
Then, since the $C$-norm of $u(t)$ is uniformly bounded with respect to $\eb\to0$ and $t\in\R$, for sufficiently small $\eb>0$, we may fix $\alpha>0$ large enough and  $\beta>0$ small enough (both independent of $\eb\to0$) such that both quadratic forms $E^1_\eb(\xi_{u_n})$ and $Q(\xi_{u_n})$ will
be positive definite. Then, passage to the weak limit gives us that
\begin{multline}
\int_{-\infty}^0e^{\beta s}(L(\xi_{u}(s))+N(\xi_{u}(s))+\beta I(u(s)))\,ds\le\\\le \liminf_{n\to\infty}
\int_{-t_n}^0e^{\beta s}(L(\xi_{u_n}(s))+N(\xi_{u_n}(s))+\beta I(u_n(s)))\,ds
\end{multline}
and, therefore
\begin{multline}
\limsup_{n\to\infty}E^1_\eb(\xi_{u_n}(0))+\\+2\int_{-\infty}^0e^{\beta s}(L(\xi_{u}(s))+N(\xi_{u}(s))+\beta I(u(s)))\,ds\le 2I(u(0)).
\end{multline}
Comparing this with the energy identity for the limit solution $u(t)$, we infer
$$
\limsup_{n\to\infty}E^1_\eb(\xi_{u_n}(0))\le E^1(\xi_u(0))\le\liminf_{n\to\infty}E^1_\eb(\xi_{u_n}(0)),
$$
where the second inequality follows from the weak convergence~\eqref{weakxi} and the fact that $E^1_\eb(\xi)$ is positive definite. Therefore,
$$
\lim_{n\to\infty}E^1_\eb(\xi_{u_n}(0)=E^1_\eb(\xi_u(0))
$$
and, indeed, $\xi_{u_n}(0)\to\xi_u(0)$ strongly in $E^1$. Thus, the asymptotic compactness is verified and the proposition is proved.
\end{proof}
We now turn to study the singular limit $\eb\to0$. We start with the convergence of individual trajectories. Let $u_\eb(t)$ and $u_0(t)$ be the solutions of hyperbolic Navier-Stokes equation \eqref{1} and the limit classical Navier-Stokes system \eqref{4}. Let also
\begin{equation}\label{33.comp}
u_\eb\big|_{t=0}=u_0\big|_{t=0}
\end{equation}
and $w(t)=w_\eb(t):=u_\eb(t)-u_0(t)$. Then, as actually established in the proof of Theorem \ref{Th1.main},
\begin{equation}\label{33.dist}
\|w_\eb(t)\|_{H}^2\le C\eb\|\xi_{u_\eb}(0)\|^2_{\mathcal E^1_\eb}e^{Kt},
\end{equation}
where the constants $C$ and $K$ are independent of $\eb\to0$, see also \cite{Hach} for similar estimates.
However, this estimate is far from being optimal. In order to improve it, we add the first boundary layer term at $t=0$ and write
\begin{equation}\label{33.bl}
u_\eb(t)=u_0(t)+\eb(\Dt u_\eb(0)-\Dt u_0(0))(1-e^{-\frac t\eb})+\bar w_\eb(t).
\end{equation}
where the value of $\Dt u_0(0)$ is determined by $u_0(0)$ via equation \eqref{4}. Then, the following result holds.
\begin{proposition}\label{Prop33.4} Let the initial data $\xi_{u_\eb}(0)$ satisfy the assumptions of Theorem \ref{Th1.main}. Then the remainder $\bar w(t)=\bar w_\eb(t)$ satisfies the following estimate:
\begin{equation}\label{33.good}
\|\xi_{\bar w}(t)\|_{\mathcal E_\eb^{-1}}\le \eb Q(\|\xi_{u_\eb}(0)\|_{\mathcal E_\eb^{1}})e^{Kt},
\end{equation}
where the constant $K$ and the function $Q$ are independent of $\eb\to0$.
\end{proposition}
\begin{proof}Indeed, the remainder $\bar w$ solves the equation
\begin{multline}
\eb\Dt^2\bar w+\Dt\bar w+A\bar w=-P[\divv(u_\eb\otimes \bar w)+\divv(\bar w\otimes u_0)]-\\-\eb P[ \divv(u_\eb\otimes w_l)+\divv(w_l\otimes u_0)-Aw_l+\Dt^2 u_0], \ \ \xi_{\bar w}\big|_{t=0}=0,
\end{multline}
where $w_l(t):=(\Dt u_\eb(0)-\Dt u_0(0))(1-e^{-\frac t\eb})$. Multiplying this equation by $\Dt A^{-1}\bar w$ and using the obvious estimate
\begin{multline}
|(P\divv(u_1\otimes u_2),\Dt A^{-1}\bar w)|\le\\\le C\|u_1\otimes u_2\|_{L^2}\|\Dt\bar w\|_{H^{-1}}
\le C\|u_1\|_{L^\infty}\|u_2\|_{L^2}\|\xi_w\|_{E^{-1}_\eb}^{1/2}
\end{multline}
together with the facts
that $u_\eb$ and $u_0$ are uniformly bounded in $C$ as well as $w_l$  is uniformly bounded in $H$, we end up with the following inequality
$$
\frac d{dt}\|\xi_{\bar w}\|^2_{E^{-1}_\eb}\le K\|\xi_{\bar w}\|^2_{E^{-1}_\eb}+\eb(w_l,\Dt \bar w)+\eb^2\|\Dt^2 u_0\|_{H^{-1}}
$$
for some positive constant $K$ depending on the initial data, but being independent of $\eb\to0$. So, we only need to estimate the second term in the right-hand side. To this end, we integrate by parts  and arrive at
\begin{multline}
\frac d{dt}\(\|\xi_{\bar w}\|^2_{E^{-1}_\eb}-\eb(w_l,\bar w)\)-K\(\|\xi_{\bar w}\|^2_{E^{-1}_\eb}-\eb(w_l,\bar w)\)\le\\\le \eb\|\Dt w_l\|_H\|\bar w\|_H+K\eb\|w_l\|_H\|\bar w\|_H+\eb^2\|\Dt^2 u_0\|^2_{H^{-1}}
\end{multline}
which in turn gives
\begin{multline}
\frac d{dt}\(\|\xi_{\bar w}\|^2_{E^{-1}_\eb}-\eb(w_l,\bar w)\)-(K+C\|\Dt w_l\|_H)\(\|\xi_{\bar w}\|^2_{E^{-1}_\eb}-\eb(w_l,\bar w)\)\le\\\le C\eb^2(\|\Dt u_0\|^2_{H^{-1}}+\|w_l\|^2_H+\|\Dt w_l\|_H).
\end{multline}
Integrating this inequality and using the facts that we have the uniform control of the
  $L^2(V^*)$-norm of $\Dt^2 u_0$ and the $L^1(H)$-norm of $\Dt w_l$, we get
\begin{equation}
\|\xi_{\bar w}(t)\|_{E^{-1}_\eb}\le \eb Q(\|\xi_{u_\eb}(0)\|_{\mathcal E^{1}_\eb})e^{Kt}.
\end{equation}
Thus, it only remains to estimate the $H^{-2}$-norm of $\Dt\bar w(t)$. This can be done exactly as at the end of the proof of Theorem \ref{Th1.main} using the fact that $\Dt^2u_0$ is bounded in $H^{-2}$ and the proposition is proved.
\end{proof}
Our next task is to compare the attractors $\mathcal A_\eb$ of hyperbolic Navier-Stokes system \eqref{1} with the global attractor of the limit equation \eqref{4}. To do this, we note that the solution operator of the limit equation is defined on a different space (since the initial data $\Dt u\big|_{t=0}$ is not required for solving the limit parabolic equation). To overcome this difficulty, we introduce following the standard scheme the phase space of the limit problem as follows:
\begin{equation}
\Phi_0:=\bigl\{(u_0,u_1)\in \mathcal{E}^1_0,\ u_1=-Au_0-P(u_0,\Nx)u_0+g\bigr\}
\end{equation}
where $\mathcal{E}^1_0=D(A)\times\mathcal{H}$,
and introduce the solution semigroup $S_0(t):\Phi_0\to\Phi_0$ by the natural expression
$$
S_0(t)(u_0(0),\Dt u_0(0)):=(u_0(t),\Dt u_0(t)),
$$
where $u_0(t)$ is the solution of the Navier-Stokes problem \eqref{4}.
\par
It is well-known that the semigroup $S_0(t)$ associated with the classical Navier-Stokes equation possesses a global attractor $\mathcal A_0$ in $\Phi_0$ which is related with the usual attractor $\bar{\mathcal A}_0$ of the Navier-Stokes problem in the phase space $H^2$ via the following expression:
\begin{equation}
\mathcal A_0=\bigl\{(u_0,u_1),\ u_0\in\bar{\mathcal A}_0,\ u_1=-Au_0-P(u_0,\Nx)u_0+g\bigr\}.
\end{equation}
The next proposition gives the strong convergence of the attractors $\mathcal A_\eb$ to the limit attractor $\mathcal A_0$.

\begin{proposition}\label{Prop33.lim} The family of attractors $\mathcal A_\eb$ is upper semi continuous at $\eb=0$ in the topology of the space $E^1_0$, i.e., for every neighbourhood $\mathcal O(\mathcal A_0)$ of the limit attractor $\mathcal A_0$, there exists $\eb_0=\eb_0(\mathcal O)>0$ such that
\begin{equation}\label{33.convatr}
\mathcal A_\eb\subset \mathcal O(\mathcal A_0)
\end{equation}
for all $\eb<\eb_0$.
\end{proposition}
\begin{proof} Indeed, according to the general theory, see e.g. \cite{BV}, it is sufficient to show that for every sequence $\eb_n\to0$ and every sequence $\xi_{u_n}\in\mathcal K_{\eb_n}$ there is a subsequence $\xi_{u_{n_k}}$ which is convergent in $C_{loc}(\R,E^1_0)$ to some $\xi_{u_0}\in\mathcal K_{0}$. Since according to the dissipative estimate, the sequence $\xi_{u_n}$ is uniformly bounded in $C_b(\R,E^1_0)$, we may assume without loss of generality that
\begin{equation}
\xi_{u_n}\rightharpoondown \xi_{u_0}
\end{equation}
in the space $L^\infty_{loc}(\R,E^1_0)$ and passing to the weak limit $n\to\infty$ in the equations for $u_n(t)$, we see that $\xi_{u_0}\in\mathcal K_0$. It remains to note that the strong convergence can be derived from the weak convergence using the energy method described in the proof of Proposition \ref{Prop33.statr}. Thus, the proposition is proved.
\end{proof}
\section{Generalizations and concluding remarks}\label{s4}

In this section, we briefly discuss possible generalizations of the obtained results and related topics. We start with the comments concerning the 3D case.

\subsection{Hyperbolic relaxation of the NS equations in 3D} Note that our estimates related with the closeness of the solutions of the relaxed problem and the initial parabolic one are based on the embedding $H^2(\Omega)\subset C(\Omega)$ and for this reason work in 3D case as well. The principal difference is related with the initial equation itself. Indeed, in contrast to the 2D case, we cannot solve globally the NS problem in 3D in the class of strong solutions, so we need to postulate it. This does not allow us to iterate estimate \eqref{21} and get the global existence of a solution for the perturbed system, so we may guarantee only the local result stated in the next proposition.

\begin{proposition}\label{Prop4.3D} Let $\Omega\subset\R^3$ be a bounded domain with sufficiently smooth boundary and let $u_0\in H^2$ be such that there exists a global strong solution $u_0(t)\in H^2$, $t\ge0$, of the classical Navier-Stokes problem \eqref{4}. Then, for every $T>0$ and $R>0$, there exists $\eb_0=\eb_0(T,R)>0$ such that for all $\eb\le\eb_0$ and all initial data $u_0'\in H^1$ such that $\|\{u_0,u_0'\}\|_{\mathcal E^1_\eb}\le R$ there exist a unique strong solution $u_\eb(t)$, $t\in[0,T]$ of problem \eqref{1} satisfying the initial condition
$$
\xi_{u_\eb}\big|_{t=0}=\{u_0,u_0')\}.
$$
Moreover, the analogue of Proposition \ref{Prop33.4} holds on the time interval $t\in[0,T]$.
\end{proposition}
The proof actually repeats the one given above for the 2D case with non-essential minor corrections, so we left it to the reader.
\begin{remark}
We also mention that there is an exceptional case where we expect  the global existence of solutions near $u_0(t)$ for the relaxed problem. Namely, this is the case where this limit solution is asymptotically stable (in a proper sense). Indeed, in this case we may establish the global existence of strong solutions for 3D Navier-Stokes equations in a {\it neighbourhood} of the solution $u_0(t)$ (using, say, the standard arguments related with the implicit function theorem). After that
it should be possible to iterate the analogue of estimate \eqref{21} and get the global existence of solutions of the relaxed equations in the neighbourhood of $u_0(t)$ if $\eb>0$ is small enough. We return to this problem somewhere else.
\end{remark}

\subsection{The case of 1D Burgers equation}\label{s4.1}
We now discuss a possible blow up of solutions of the hyperbolic
Navier-Stokes equations for $\eb>0$ and arbitrarily large initial data. We start with more simple hyperbolic relaxation of 1D Burgers equation
\begin{equation}\label{2.bure}
\eb\Dt^2 u+\Dt u+\partial_x(u^2)=\partial^2_x u+g,\ \ x\in[0,L],\ \ \xi_u\big|_{t=0}=\xi_0
\end{equation}
endowed by the Dirichlet boundary conditions. Then, on the one hand the arguments given in Section \ref{s3}, we see that the analogue of Theorem \ref{Th1.main} holds. Namely, for every $R>0$, there exists $\eb_0=\eb_0(R)$ such that a unique global solution $\xi_{u}(t)\in\mathcal E^1_\eb$ of problem \eqref{2.bure} exists for all $\eb\le\eb_0$ and the initial data $\xi_u(0)$ such that $\|\xi_u(0)\|_{\mathcal E^1_\eb}\le R$. Moreover, this solution satisfies the dissipative estimate \eqref{3}.
\par
On the other hand, this model is essentially simpler than the original Navier-Stokes problem and we are able to prove here that the solutions may blow up if the initial energy is large enough. To see this, we fix $\eb=1$ (which we may always assume due to scaling) and consider
the following hyperbolic equation on the whole line $x\in\R$:
\begin{equation}\label{2.bur}
\Dt^2 u+\Dt u+\partial_x(u^2)=\partial^2_x u,\ \ x\in\R,\ \ \xi_u\big|_{t=0}=\xi_0.
\end{equation}
Actually this equation possesses the finite propagation speed  property, so the boundary conditions are also not essential since the blowing up solution which we construct will be localized in space.
Then, the following result can be proved.

\begin{proposition} There exist smooth finitely supported initial data $\xi_u(0)$ such that the corresponding solution of \eqref{2.bur} blows up in finite time.
\end{proposition}
\begin{proof} Also this result is well-known (see e.g., \cite{sider}), for the convenience of the reader we reproduce key points of the proof here. We start with the finite propagation property which is the key technical tool for verifying the blow up.
\par
 Namely, if the support of the initial data $\xi_u(0)$ satisfies
 $$
 \operatorname{supp}\xi_u(0)\subset [-R,R],
 $$
  then
\begin{equation}\label{2.sup}
\operatorname{supp}\xi_u(t)\subset [-R-t,R+t].
\end{equation}
To verify this, following~\cite{Evans, Hach} (see also references therein),
we consider the cone $K$ in the    $(x,t)$-plane with base $-R\le x\le R$ and
vertex $t=R$, $x=0$. Let, for $0\le t\le R$,
\begin{samepage}
$$
e(t):=\int_{-(R-t)}^{R-t}\left((\partial_t u(t,x))^2+(\partial_x u(t,x))^2\right)dx
$$
 be the energy at the section of the cone $K$ at time $t$. Then
\end{samepage}

$$
\aligned
&\frac d{dt}e(t)=2\int_{-(R-t)}^{R-t}\left(\partial_t u\partial_{tt} u+\partial_x u\partial_{xt} u\right)dx-\\
&\qquad -[((\partial_t u)^2+(\partial_x u)^2)\vert_{(t,R-t)}+((\partial_t u)^2+(\partial_x u)^2)\vert_{(t,-(R-t))}]=\\
&2\int_{-(R-t)}^{R-t}\left(\partial_t u(\partial_{tt} u- \partial_{xx}u) \right)dx-\\
&-[((\partial_t u)^2+2\partial_t u\partial_x u +u(\partial_x u)^2)]\vert_{(t,R-t)}-\\
&\qquad\qquad-[((\partial_t u)^2+2\partial_t u\partial_x u+(\partial_x u)^2)]\vert_{(t,-(R-t))}\le\\
&2\int_{-(R-t)}^{R-t}\left(\partial_t u(\partial_{tt} u- \partial_{xx}u) \right)dx=
2\int_{-(R-t)}^{R-t}\left(\partial_t u(-\partial_{t} u- 2u\partial_{x}u) \right)dx\le\\
&4\int_{-(R-t)}^{R-t}u\partial_t u \partial_{x}u dx\le 2\|u\|_{L^\infty}e(t).
\endaligned
$$
Now, if $\operatorname{supp}\xi_u(0)\cap [-R,R]=\emptyset$, then $e(0)=0$ and, hence,
$e(t)=0$ for $0\le t\le R$. Therefore the solution vanishes in the cone $K$,
which is equivalent to~\eqref{2.sup}.
\par
We are now ready to verify the blow up. Following \cite{sider} (see also \cite{Hach}),
 multiply equation \eqref{2.bur} by $e^{-x}$ and integrate over $x\in\R$. Then, denoting
  $$
  y(t):=\int_Re^{-x}u(t,x)\,dx,
  $$
  we get
\begin{equation}\label{2.bl1}
y''(t)+y'(t)=y(t)+\int_Re^{-x}u^2(t,x)\,dx.
\end{equation}
All of the terms in this equation will be finite if we start from the initial data $\xi_u(0)$ with finite support due to the property \eqref{2.sup}. Let us fix an arbitrary $T>0$ and assume that $\operatorname{supp}\xi_u(0)\subset[-1,1]$. Then, due to \eqref{2.sup} and Jensen inequality, we can estimate the nonlinear term in \eqref{2.bl1}
\begin{equation}
\int_Re^{-x}u^2(t,x)\,dt\ge e^{-T-1}\int_\R(e^{-x}u(t,x))^2\,dx\ge e^{-T-1}(2(T+1)^{-1}y^2(t)
\end{equation}
which gives
\begin{equation}
y''(t)+y'(t)\ge y(t)+e^{-T-1}(2(T+1)^{-1}y^2(t),\ \ t\in[0,T]
\end{equation}
which guarantees the blow up of solutions of \eqref{2.bur}
if the localized initial data is large enough, see \cite{sider} for details. Thus, the proposition is proved.
\end{proof}

\subsection{Finite propagation approximation of the Navier-Stokes problem}
  We emphasize once more that the finite prorogation speed property is
  crucial for this method. Since for the initial hyperbolic Navier-Stokes
  system this property clearly fails due to the presence of the non-local
  pressure term, the method is not applicable to these equation and one
  should find an alternative way to establish the finite time blow up. However, if we modify slightly the approximation scheme for the original Navier-Stokes problem, we may get the finite propagation speed property. Namely, let us consider the following problem:
  \begin{equation}\label{2.alphaeb}
  \eb\Dt^2u+\Dt u+\divv(u\otimes u)=\Dx u+\frac1\alpha\Nx\divv u+g,\ \ u\big|_{\partial\Omega}=0,
  \end{equation}
where $\alpha>0$ is one more small parameter. Then, on the one hand, as not difficult to see, the solutions of \eqref{2.alphaeb} converge as $\alpha\to0$ to the corresponding solutions of the hyperbolic Navier-Stokes problem \eqref{1}. On the other hand, equation \eqref{2.alphaeb} possesses the finite propagation speed property, see \cite{Hach} for the details. When the finite propagation speed property is established, the blow up of smooth solutions for problem \eqref{2.alphaeb} can be verified exactly as for the case of Burgers equation, see \cite{Sch} for the details.

\subsection{Connection with viscoelastic fluids} Note that the existence of blow up solutions discussed is related mainly with the fact  that the hyperbolic relaxation \eqref{1} of the Navier-Stokes equations does not possess a reasonable energy functional for $\eb>0$. By this reason, the model \eqref{1} looks a bit non-physical. Fortunately, this drawback can be easily corrected by adding an extra small term to \eqref{1} which does not destroy the hyperbolic structure of the equations and the estimates obtained above, but restores the energy identity. Indeed, let us consider the following particular case of the so-called Jeffrey model for viscoelastic fluids:
\begin{equation}\label{2.old}
\begin{cases}
\Dt u+(u,\Nx)u+\Nx p=\divv\sigma+g,\ \ \divv u=0,\\ \eb\Dt\sigma+\sigma=\gamma,
\end{cases}
\end{equation}
where $\gamma:=\frac12\(\Nx u+\Nx^*u\)$ is a strain rate tensor, see \cite{const,GGP} for more details. Then, integrating the second equation in time, we may write \eqref{2.old} as an Euler equation with memory term
$$
\Dt u+(u,\Nx)u+\Nx p=\int_{-\infty}^t\beta(t-s)\Dx u(s)\,ds+g,\ \ \beta(s)=\frac1\eb e^{-\frac{s}\eb}.
$$
Alternatively, excluding $\sigma$ from the first equation by differentiating it in time and using the second one gives
\begin{equation}\label{2.J}
\eb\Dt^2u+\eb\Dt[(u,\Nx)u]+\Dt u+(u,\Nx)u+\Nx p=\Dx u+g
\end{equation}
which coincides with \eqref{1} up to the desired extra term $\eb\Dt[(u,\Nx)u]$. This extra term allows to restore the energy identity. Namely, multiplying the first equation of \eqref{2.old} by $u$ and integrating over $x$, we get the energy identity of the form
\begin{equation}\label{2.energy}
\frac12\frac d{dt}\(\|u\|^2_{L^2}+\eb\|\divv\sigma\|^2_{L^2}\)+\|\divv \sigma\|^2_{L^2}=(u,g).
\end{equation}
This identity guarantees at least the global existence of weak solutions and destroys the blow up mechanism described above. We return to this problem somewhere else.

\end{document}